\documentclass[a4paper]{article}
\usepackage{amsmath,amsthm,amssymb,enumerate,mathrsfs}
\newtheorem{dfn}{Definition}[section]
\newtheorem{thm}[dfn]{Theorem}

\newtheorem{lem}[dfn]{Lemma}

 \makeatletter
 \@addtoreset{equation}{section} 
\makeatother
\newcommand{\N}{\mathbb{N}}

\newcommand{\R}{\mathbb{R}}

\newcommand{\A}{\mathcal{A}}

\newcommand{\Hom}{\mathrm{Hom}}
\newcommand{\I}{\hbox{I}}
\newcommand{\II}{\hbox{II}}
\newcommand{\III}{\hbox{III}}
\newcommand{\IV}{\hbox{IV}}

\def\Xint#1{\mathchoice
 {\XXint\displaystyle\textstyle{#1}}%
 {\XXint\textstyle\scriptstyle{#1}}%
 {\XXint\scriptstyle\scriptscriptstyle{#1}}%
 {\XXint\scriptscriptstyle\scriptscriptstyle{#1}}%
 \!\int} \def\XXint#1#2#3{{\setbox0=\hbox{$#1{#2#3}{\int}$}
 \vcenter{\hbox{$#2#3$}}\kern-.5\wd0}} 
\def\dashint{\Xint{-}}
\makeatletter
 \def\@maketitle{%
\begin{flushright}%
{\large \@date}%
\end{flushright}%
\par\vskip 1.5em
\begin{center}%
{\LARGE \@title \par}%
\end{center}%
\begin{center}%
{\large \@author}%
\end{center}%
\par\vskip 1.5em
 }
\title{
Partial regularity result of elliptic systems with Dini continuous coefficients and $q$-growth}
\author{Taku Kanazawa \\
Graduate School of Mathematics \\
Nagoya University, JAPAN
}
\date{Ver. 5 July 2013
}
\begin{document}

\maketitle

\begin{center}
\begin{minipage}[t]{10cm}
\small{
\noindent \textbf{Abstract.}
We establish partial regularity result for vector-valued solutions $u:\Omega\to\R^N$ to second order elliptic systems
of the type:
\[
 -\mathrm{div}(A(x,u,Du))=f(x,u,Du) \qquad \mathrm{in}\>\Omega ,
\]
where the coefficients $A:\Omega\times\R^N\times\Hom(\R^n,\R^N)\to\Hom(\R^n,\R^N)$ satisfies Dini condition respect to 
$(x,u)$ with growth order $q\geq 2$. We prove $C^1$-regularity of the solutions outside of singular sets.
\medskip

\medskip

\noindent \textbf{Mathematics~Subject~Classification~(2010):}
35J60, 35B65.

}
\end{minipage}
\end{center}
\section{Introduction}
In this paper, we consider the second order nonlinear elliptic systems in divergence form of the following type:
\begin{equation}
 -\mathrm{div}(A(x,u,Du))=f(x,u,Du) \qquad \text{in}\>\Omega .
 \label{system}
\end{equation}
Here, $\Omega$ is a bounded domain in $\R^n$, $u$ takes values in $\R^N$ with coefficients 
$A:\Omega\times\R^N\times\Hom(\R^n,\R^N)\to\Hom(\R^n,\R^N)$. 

The regularity theory with the growth of $A(x,\xi,p)$ with respect to $p$ has been proved by Giaquinta and Modica \cite{GMo2}. 
They proved that weak solutions of \eqref{system} has H\"{o}lder continuous first derivatives outside of a singular set of 
Lebesgue measure zero if $(1+\lvert p \rvert)^{-1}A(x,\xi,p)$ is H\"{o}lder continuous in variables $(x,\xi)$ uniformly with respect to $p$. 
In \cite{DG}, Duzaar and Grotowski gave a simplified proof of their result without 
$L^q$-$L^2$ estimates for $Du$. The method of proof also gives the optimal result in one step,
i.e. if $(1+\lvert p \rvert)^{-1}A(x,\xi,p)$ is in $C^{0,\alpha}$ for some $0<\alpha<1$ in $(x,\xi)$ then $u$ is in 
$C^{1,\alpha}$ outside of the singular set. The essential feature is the use of the $\A$-harmonic approximation lemma
(cf. \cite[Lemma 2.1]{DG}; see also Lemma \ref{A-harm}). 

Duzaar and Gastel \cite{DGa} prove under weaker assumptions on $A(x,\xi,p)$ with respect to continuity in the variables
$(x,\xi)$. More precisely, they assume for the continuity of $A(x,\xi,p)$ with respect to the variables $(x,\xi)$ that 
\begin{equation}
 \lvert A(x,\xi,p)-A(\widetilde{x},\widetilde{\xi},p)\rvert 
 \leq \kappa (\lvert\xi\rvert)\mu\left( \lvert x-\widetilde{x}\rvert +\lvert\xi-\widetilde{\xi}\rvert\right)(1+\lvert p \rvert), 
 \label{DGaassumption}
\end{equation}
for all $x,\widetilde{x}\in\Omega$, $\xi,\widetilde{\xi}\in\R^N$, $p\in\Hom(\R^n,\R^N)$, where 
$\kappa :[0,\infty)\to[1,\infty)$ is nondecreasing, and $\mu:(0,\infty)\to[0,\infty)$ is nondecreasing and
concave with $\mu(0+)=0$. They also have to require that $r\mapsto r^{-\alpha}\mu(r)$ is nonincreasing for 
some $0<\alpha<1$, and that 
\begin{equation}
 F(r)=\int_0^r\frac{\mu(\rho)}{\rho}d\rho<\infty \quad \text{for some }r>0. \label{originDini}
\end{equation}
They conclude that a bounded weak solution of elliptic system \eqref{system} satisfying \eqref{DGaassumption} and 
(\ref{originDini}) is in $C^1$ outside a closed singular set with Lebesgue measure zero. 

The condition \eqref{originDini} is called Dini condition in the literature, although Dini himself \cite{Di} used 
a slightly weaker conditions a century ago. It had some significance for the theory of linear elliptic partial 
differential equations in the first half of the century, cf. \cite{HW}. 

Qiu \cite{Qiu} extend the result in \cite{DGa}, which is the result under quadratic growth condition, 
to the subquadratic case. In this case, the assumptions \eqref{DGaassumption} and \eqref{originDini} are modified as 
\[
 \lvert A(x,\xi,p)-A(\widetilde{x},\widetilde{\xi},p)\rvert 
 \leq \kappa (\lvert\xi\rvert) \mu\left( \lvert x-\widetilde{x}\rvert^q +\lvert\xi-\widetilde{\xi}\rvert^q\right)
 (1+\lvert p \rvert)^{-2/q},
\]
and 
\[
 F(r)=\int_0^r\frac{\sqrt{\mu(\rho)}}{\rho}d\rho<\infty \quad \text{for some }r>0,
\]
where $1<q<2$. 

In this paper, we consider the regularity theory in case of superquadratic, i.e. $q\geq 2$. Thus, we assume the continuity of 
$A(x,\xi,p)$ with respect to the variables $(x,\xi)$ that 
\[
 \lvert A(x,\xi,p)-A(\widetilde{x},\widetilde{\xi},p)\rvert 
 \leq \kappa (\lvert\xi\rvert) \mu\left( \lvert x-\widetilde{x}\rvert +\lvert\xi-\widetilde{\xi}\rvert\right)
 (1+\lvert p \rvert)^{q-1},
\]
and to obtain the regularity result, we assume the modified Dini condition such that 
\[
 F(r)=\int_0^r\frac{\mu^\beta(\rho)}{\rho}d\rho<\infty \quad \text{for some}\ r>0\ \text{and}\ \beta\in \left( 0,1\right].
\]
Under these assumptions and $q$-growth condition for 
inhomogeneous term, we obtain that a bounded weak solution of \eqref{system} is $C^1$(see Theorem \ref{pr}). 

Our result is different from the result of Qiu \cite{Qiu2}. The main difference is the version of $\A$-harmonic
approximation lemma which we used. Lemma 2.1 in \cite{Qiu2} (see also \cite[Lemma 2.1]{DG}) only guarantee the existence of 
$\A$-harmonic function $h$ which approximate the rescaled solution $w$ in $L^2$. This restricts the growth order $q<n$ to estimate 
$\dashint_{B_\rho(x_0)}\lvert w-h \rvert^qdx$ by the Sobolev-Poincar\'{e} inequality. In contrast our $\A$-harmonic 
approximation lemma guarantees the approximation in $L^q$ as well as in $L^2$, and this allows us to obtain the regularity 
result at any growth order. 

We close this section by briefly summarizing the notation used in this paper.
As note above, we consider a bounded domain $\Omega\subset\R^n$, and maps from $\Omega$ to $\R^N$,
where we take $n\geq 2$, $N\geq 1$. For a given set $X$ we denote by $\mathscr{L}^n(X)$ as $n$-dimensional 
Lebesgue measure. We write $B_\rho(x_0):=\{ x\in\R^n\> :\> \lvert x-x_0 \rvert<\rho\}$. For bounded set
$X\subset\R^n$ with $\mathscr{L}^n(X)>0$, we denote the average of a given function $g\in L^1(X,\R^N)$ by 
$\dashint_Xgdx$ i.e., $\dashint_Xgdx =\frac{1}{\mathscr{L}^n(X)}\int_Xgdx$. In particular, we write
$g_{x_0,\rho}=\dashint_{B_\rho(x_0)\cap\Omega}gdx$. We write $\mathrm{Bil}(\mathrm{Hom}(\R^n,\R^N))$
for the space of bilinear forms on the space $\mathrm{Hom}(\R^n,\R^N)$ of linear maps from $\R^n$ to
$\R^N$. We denote $c$ a positive constant, possibly varying from line by line. Special occurrences will be denoted by 
capital letters $K$, $C_1$, $C_2$ or the like.  

\section{Hypothesis and Statement of Results}
\begin{dfn}\label{wsol}
We define $u\in W^{1,q}(\Omega,\R^N),q\geq 2$ is a weak solution of \eqref{system} if $u$ satisfies
\begin{equation}
 \int_\Omega \langle A(x,u,Du),D\varphi\rangle dx
 =\int_\Omega \langle f,\varphi\rangle dx \label{ws}
\end{equation}
for all $\varphi\in C^{\infty}_0(\Omega,\R^N)$, where $\langle\cdot,\cdot\rangle$ is the standard Euclidean 
inner product on $\R^N$ or $\R^{nN}$.
\end{dfn}
We assume following structure condition.
\begin{enumerate}
\item[\bf{(H1)}]
$A(x,\xi,p)$ is differentiable in $p$ with continuous derivatives. Moreover, there exists $L\geq 1$ such that
\[
 \left\lvert D_pA(x,\xi,p)\right\rvert\leq L(1+\lvert p \rvert)^{q-2} 
 \quad \text{for all }(x,\xi,p)\in\Omega\times\R^N\times\mathrm{Hom}(\R^n,\R^N);
\]
this infers the existence of a modulus of continuity $\omega:[0,\infty)\times [0,\infty)\to [0,1]$ with 
$\omega(t,0)=0$ for all $t$ such that $t\mapsto\omega(s,t)$ is nondecreasing for fixed $s$,
$s\mapsto\omega(s,t)$ is concave and nondecreasing for fixed $t$. $\omega(s,t)$ also satisfies 
\[
 \left\lvert D_p A(x,\xi,p)-D_p A(\widetilde{x},\widetilde{\xi},\widetilde{p})\right\rvert 
 \leq L\omega\left(\lvert\xi\rvert+\lvert\nu\rvert,\lvert x-\widetilde{x}\rvert^2 
 +\lvert\xi-\widetilde{\xi}\rvert^2+\lvert p-\widetilde{p}\rvert^2\right) 
 (1+\lvert p \rvert+\lvert\widetilde{p}\rvert)^{q-2}.
\] 
for all $(x,\xi,p),(\widetilde{x},\widetilde{\xi},\widetilde{p})\in\Omega\times\R^N\times\mathrm{Hom}(\R^n,\R^N)$ with
$\lvert\xi\rvert+\lvert p \rvert\leq M$.
\item[\bf{(H2)}]
$A(x,\xi,p)$ is uniformly strongly elliptic i.e., for some $\lambda>0$, $A(x,\xi,p)$ satisfies
 \[
  \left\langle D_pA(x,\xi,p)\nu,\nu\right\rangle \geq \lambda \lvert\nu\rvert^2(1+\lvert p \rvert)^{q-2}
  \quad \text{for all }x\in\Omega, \xi\in\R^N, p,\nu\in\mathrm{Hom}(\R^n,\R^N);
 \]
\item[\bf{(H3)}]
 There exists a modulus of continuity $\mu:[0,\infty)\to(0,\infty)$, and a nondecreasing function
$\kappa :[0,\infty)\to[1,\infty)$ such that 
\begin{equation}
 \lvert A(x,\xi,p)-A(\widetilde{x},\widetilde{\xi},p)\rvert 
 \leq \kappa (\lvert\xi\rvert)\mu\left( \lvert x-\widetilde{x}\rvert+\lvert\xi-\widetilde{\xi}\rvert\right)
 (1+\lvert p \rvert)^{q-1}
\end{equation}
for all $x,\widetilde{x}\in\Omega$, $\xi,\widetilde{\xi}\in\R^N$, $p\in\mathrm{Hom}(\R^n,\R^N)$. Without loss of generality 
we may assume that
\begin{enumerate}
\item[($\mu$1)] $\mu$ is nondecreasing function with $\mu(+0)=0$.
\item[($\mu$2)] $\mu$ is concave; in the proof of the regularity theorem we have to require that $r \mapsto r^{-\alpha}\mu(r)$ 
 is nonincreasing for some exponent $\alpha\in(0,1)$.
\end{enumerate}
We also require modified Dini's condition: 
\begin{enumerate}
\item[($\mu$3)] $\displaystyle F(r) := \int_0^r \frac{\mu^\beta(\rho)}{\rho}d\rho <+\infty$ for some $r>0$ and 
 $\displaystyle \beta\in\left( 0,1\right]$. 
\end{enumerate}
\item[\bf{(H4)}]
There exists constants $a$ and $b$, with $a$ possibly depending on $M>0$, such that 
 \[
  \lvert f(x,\xi,p)\rvert\leq a(M)\lvert p \rvert^q +b
 \]
for all $x\in\Omega$, $\xi\in\R^N$ with $\lvert\xi\rvert\leq M$, and $p\in\mathrm{Hom}(\R^n,\R^N)$.
\end{enumerate}

Using above structure conditions, we state our main theorem.

\begin{thm}\label{pr}
Let $u\in W^{1,q}(\Omega,\R^N)\cap L^\infty(\Omega,\R^N)$ be a bounded weak solution to \eqref{system} under the structure conditions 
{\rm (\textbf{H1}), (\textbf{H2}), (\textbf{H3}), (\textbf{H4}), ($\mu$1), ($\mu$2)} and {\rm ($\mu$3)}, 
satisfying $\lVert u\rVert_\infty\leq M$ and $2^{(10-9q)/2}\lambda >a(M)M$.
Then there is a relatively closed set $\mathrm{Sing}\, u\subset\Omega$, such that the weak solution $u$ satisfies 
$u\in C^1(\Omega\setminus\mathrm{Sing}\, u,\R^N)$.
Further $\mathrm{Sing}\, u \subset \Sigma_1\cup\Sigma_2$, where 
\begin{align*}
 \Sigma_1:&=\left\{ x_0\in\Omega \>\> :\>\>
 \liminf_{\rho\searrow 0}\dashint_{B_\rho(x_0)}\lvert Du-(Du)_{x_0,\rho}\rvert^q dx>0\right\}, \quad \text{and}\\
 \Sigma_2:&=\left\{ x_0\in\Omega \>\> :\>\>
 \limsup_{\rho\searrow 0}\lvert (Du)_{x_0,\rho}\rvert =\infty\right\}
\end{align*}
and in particular, $\mathscr{L}^n(\mathrm{Sing}\, u)=0$. In addition, for $\sigma\in[\alpha,1)$ and 
$x_0\in\Omega\setminus\mathrm{Sing}\, u$ the derivative of $u$ has modulus of continuity $r\mapsto r^\sigma+F(r)$ in a 
neighborhood of $x_0$.
\end{thm}
\section{Some preliminaries}
In this section we recall the $\A$-harmonic approximation lemma, and some standard estimates
for the proof of the regularity theorem. 

First we state the definition of $\A$-harmonic function and present the following version of an
$\A$-harmonic approximation lemma which can be retrieved from the corresponding parabolic version in \cite[Lemma 3.2]{DMS}. 
This lemma allowed us to approximate the weak solution $u$ to the solution of constant coefficients elliptic system in 
$L^2$ as well as in $L^q$. For more detail about $\A$-harmonic approximation technique, we refer to the survey paper \cite{DM}. 
\begin{dfn}[{\cite[Section 1]{DG}}]
For a given $\A\in\mathrm{Bil}(\mathrm{Hom}(\R^n,\R^N))$, we say that $h\in W^{1,q}(\Omega,\R^N)$ 
is an $\A$-harmonic function, if $h$ satisfies
\[
 \int_{\Omega}\A(Dh,D\varphi)dx=0
\]
for all $\varphi\in C^\infty_0(\Omega,\R^N)$.
\end{dfn}

\begin{lem}[{\cite[Lemma 2.3]{BDHS}}]\label{A-harm}
Let $0<\lambda\leq L$ and $q\geq 2$ be given. For every $\varepsilon>0$, there exists a constant 
$\delta=\delta(n,N,q,\lambda,L,\varepsilon)\in(0,1]$ such that the following holds: assume that $\gamma\in[0,1]$ and that $\A$
is a bilinear form on $\mathrm{Hom}(\R^n,\R^N)$ with the properties
\begin{equation}
 \A (\nu,\nu)\geq \lambda\lvert\nu\rvert^2, \quad \text{and} \quad 
 \A (\nu,\widetilde{\nu})\leq L\lvert\nu\rvert\lvert\widetilde{\nu}\rvert,
\end{equation}
for all $\nu,\widetilde{\nu}\in\mathrm{Hom}(\R^n,\R^N)$. Furthermore, let $w\in W^{1,q}(B_\rho(x_0),\R^N)$ be an approximately 
$\A$-harmonic map in the sense that there holds
\[
 \left\lvert\dashint_{B_\rho(x_0)}\A (Dw,D\varphi )dx \right\rvert\leq \delta\sup_{B_\rho(x_0)}\lvert D\varphi\rvert
\]
for all $\varphi\in C^\infty_0(B_\rho(x_0),\R^N)$ and that
\[
 \dashint_{B_\rho(x_0)}\{\lvert Dw\rvert^2+\gamma^{q-2}\lvert Dw\rvert^q\} dx\leq 1.
\]
Then there exists an $\A$-harmonic function $h\in C^\infty(B_{\rho/2}(x_0),\R^N)$ that satisfies
\begin{equation}
 \dashint_{B_{\rho/2}(x_0)}\left\{\lvert Dh\rvert^2+\gamma^{q-2}\lvert Dh\rvert^q\right\} dx\leq \tilde{C}(n,q)
\end{equation}
and, at the same time,
\begin{equation}
 \dashint_{B_{\rho/2}(x_0)}\left\{\left\lvert\frac{w-h}{\rho/2}\right\rvert^2
 +\gamma^{q-2}\left\lvert\frac{w-h}{\rho/2}\right\rvert^q\right\} dx\leq\varepsilon.
\end{equation}
\end{lem}
Next is a standard estimates for the solutions to homogeneous second order elliptic systems with
constant coefficients, due originally to Campanato \cite[Teorema 9.2]{Cam}.
For convenience, we state the estimate in a slightly general form than the original one.
\begin{thm}[{\cite[Theorem 2.3]{DG}}]\label{Campanato}
Consider $A$, $\lambda$ and $L$ as in Lemma \ref{A-harm}. Then there exists $C_0\geq 1$ depending only on 
$n$, $N$, $\lambda$ and $L$ such that any $\A$-harmonic function $h$ on $B_{\rho/2}(x_0)$ satisfies
\begin{equation}
 \left(\frac{\rho}{2}\right)^2\sup_{{B_{\rho/4}}(x_0)}\lvert Dh\rvert^2
 +\left(\frac{\rho}{2}\right)^4\sup_{B_{\rho/4}(x_0)}\lvert D^2h\rvert^2
 \leq C_0\left(\frac{\rho}{2}\right)^2\dashint_{B_{\rho/2}(x_0)}\lvert Dh\rvert^2dx. \label{campanato}
\end{equation}
\end{thm}
We state the Poincar\'{e} inequality in a convenient form.
\begin{lem}[{\cite[Proposition 3.10]{GM}}]\label{Poincare}
There exists $C_P\geq 1$ depending only on $n$ such that every $u\in
W^{1,q}(B_\rho(x_0),\R^N)$ satisfies
\begin{equation}
 \int_{B_\rho(x_0)}\lvert u-u_{x_0,\rho}\rvert^qdx 
 \leq C_P\rho^q\int_{B_\rho(x_0)}\lvert Du\rvert^qdx. \label{poincare}
\end{equation}
\end{lem}
Using Young's inequality, we obtain the following estimates. 
\begin{lem}[{\cite[Lemma 3.7]{Kana}}]\label{Young2}
Consider fixed $a,b\geq 0$, $q\geq 1$. Then for any $\varepsilon>0$, there exists
$K=K(q,\varepsilon)\geq 0$ satisfying
\begin{equation}
 (a+b)^q\leq (1+\varepsilon)a^q +K b^q. \label{young2}
\end{equation}
\end{lem}
\begin{lem}[{\cite[Lemma 2.1]{GMo}}]\label{GM}
For $\delta \geq 0$, and for all $a,b\in\R^{nN}$ we have 
\begin{equation}
 4^{-(1+2\delta)}\leq
 \frac{\displaystyle\int_0^1(1+\lvert sa+(1-s)b\rvert^2)^{\delta/2}ds}{(1+\lvert a\rvert^2+\lvert b-a\rvert^2)^{\delta/2}}
 \leq 4^\delta . \label{GM2}
\end{equation}
\end{lem}
In the followings, we write the modulus of continuity $\mu$ as 
\[
 \eta(t):=\mu^2\left(\sqrt{t}\right)
\]
by technical reason (cf. {\bf (H3)}). The conditions ($\mu$1), ($\mu$2) and ($\mu$3) are expressed as 
\begin{enumerate}
\item[($\eta$1)] $\eta$ is continuous, nondecreasing, and $\eta(+0)=0$, 
\item[($\eta$2)] $\eta$ is concave; and $t \mapsto t^{-\alpha}\eta(t)$ is nonincreasing for the same exponent $\alpha$ as 
in ($\mu$2), 
\item[($\eta$3)] $\displaystyle 
\widetilde{F}(t):= \left[ 2F\left(\sqrt{t}\right)\right]^2
 =\left[ \int_0^t\frac{\sqrt{\eta^\beta(\tau)}}{\tau}d\tau\right]^2<+\infty$ for some $t>0$.
\end{enumerate}
Changing $\kappa$ by a constant, but keeping $\kappa\geq 1$, we can also assume that
\begin{enumerate}
\item[($\eta$4)] $\eta(1)=1$, implying $t\leq\eta(t)\leq 1$ \quad for $t\in(0,1]$.
\end{enumerate}
From the fact that $\eta$ is nondecreasing, for $t\leq s$ and $\sigma\leq 1/\alpha$, we deduce
$s\eta^{\sigma}(t)\leq s\eta^{\sigma}(s)$.
For $s\leq t$, we use nonincreasing property of $t^{-\alpha}\eta(t)$ and $\eta(s)\leq 1$, and we obtain
$s\eta^{\sigma}(t)\leq t$. Combining both cases we obtain
\[
 s\eta^{\sigma}(t)\leq s\eta^{\sigma}(s)+t \quad \text{for}\> s\in[0,1],\ t>0,\ \sigma\leq\frac{1}{\alpha}.
\]
In particular, we have
\begin{enumerate}
\item[($\eta$5)] $s\eta(t)\leq s\eta(s)+t$ \quad for $s\in[0,1]$, $t>0$,
\item[($\eta$6)] $s\sqrt{\eta(t)}\leq s\sqrt{\eta(s)}+t$ \quad for $s\in[0,1]$, $t>0$. 
\end{enumerate}
From ($\eta 2$) we infer for $i\in\N\cup\{ 0\}$, $\theta\in(0,1/8]$, $t>0$
\[
 \int_{\theta^{2(i+1)}t}^{\theta^{2i}t}\frac{\sqrt{\eta^\beta(\tau)}}{\tau}d\tau
 \geq \sqrt{\frac{\eta^\beta(\theta^{2i}t)}{(\theta^{2i}t)^{\alpha\beta}}}\int_{\theta^{2(i+1)}t}^{\theta^{2i}t}
  \tau^{(\alpha\beta-2)/2}d\tau 
 =\frac{2}{\alpha\beta}\left( 1-\theta^{\alpha\beta}\right)\sqrt{\eta^\beta(\theta^{2i}t)} ,
\]
which implies
\begin{align}
 \sum_{i=0}^{k-1}\sqrt{\eta^\beta(\theta^{2i}t)} 
 &\leq \frac{\alpha\beta}{2(1-\theta^{\alpha\beta})}\sqrt{\widetilde{F}(t)} \label{sumeta}
\end{align}
for $k\in\N$. This yields in particular that 
\begin{equation}
 \eta(t)\leq \frac{\alpha^2\beta^2}{4(1-\theta^{\alpha\beta})^2}\widetilde{F}(t) \label{etaF}
\end{equation}
for all $t\in [0,1]$. Moreover, for $t\in [0,1]$, $\theta\in(0,1/8]$, we have
\begin{align}
 t^{-\alpha}\widetilde{F}(t)
  &=t^{-\alpha}\left[\sqrt{\widetilde{F}(\theta t)}
   +\int_{\theta t}^t\sqrt{\tau^{-\alpha}\eta(\tau)}\tau^{(\alpha-2)/2}d\tau\right]^2 \notag\\
  &\leq t^{-\alpha}\left[\sqrt{\widetilde{F}(\theta t)}
   +\frac{2}{\alpha}\sqrt{(\theta t)^{-\alpha}\eta(\theta t)}\left\{\sqrt{t^\alpha}-\sqrt{(\theta t)^\alpha}\right\} \right]^2 \notag\\
  &\leq \left[\sqrt{t^{-\alpha}\widetilde{F}(\theta t)}
   +\sqrt{(\theta t)^{-\alpha}\widetilde{F}(\theta t)}
   \frac{1-\theta^{\alpha/2}}{1-\theta^{\alpha\beta}}\right]^2 \notag\\
  &\leq 4(\theta t)^{-\alpha}\widetilde{F}(\theta t). \label{Fnoninc}
\end{align}

\section{Caccioppoli-type inequality}
For $s,t\geq 0$ let
\[
 \rho_1(s,t):= (1+t)^{-1}\kappa^{-1}(s+t), \qquad G(s,t):= (1+t)^2\kappa^{2q}(s+t).
\]
Note that $\rho_1\leq 1$ and $G\geq 1$.
\begin{lem}\label{Caccioppoli}
Consider $\nu\in\mathrm{Hom}(\R^n,\R^N)$ and $\xi\in\R^N$ with $\lvert\xi\rvert\leq M$ fixed. Let 
$u\in W^{1,q}(\Omega,\R^N)\cap L^\infty(\Omega,\R^N)$ be a bounded weak solution to \eqref{system} under 
the structure conditions {\rm (\textbf{H1}), (\textbf{H2}), (\textbf{H3}), (\textbf{H4}), ($\eta$1), ($\eta$2), ($\eta$3)} 
and {\rm ($\eta$4)} with satisfying $\lVert u\rVert_\infty \leq M$ and $2^{(10-9q)/2}\lambda >a(M)M$. 
Then for any $x_0\in\Omega$ and $\rho\leq \rho_1(\lvert\xi\rvert,\lvert\nu\rvert)$ such
that $B_\rho(x_0)\Subset\Omega$, there holds 
\begin{align}
 &\dashint_{B_{\rho/2}(x_0)}\left\{\frac{\lvert Du-\nu\rvert^2}{(1+\lvert\nu\rvert)^2}
 +\frac{\lvert Du-\nu\rvert^q}{(1+\lvert\nu\rvert)^q}\right\} dx \notag\\
 \leq C_1&\Bigg[ \dashint_{B_\rho(x_0)}\left\{\left\lvert\frac{u-\xi-\nu(x-x_0)}{\rho(1+\lvert\nu\rvert)}\right\rvert^2 
 +\left\lvert\frac{u-\xi-\nu(x-x_0)}{\rho(1+\lvert\nu\rvert)}\right\rvert^q\right\} dx 
 +G(\lvert\xi\rvert,\lvert\nu\rvert)\eta(\rho^2)
 +\left( a\lvert\nu\rvert +b\right)^2\rho^2 \Bigg] \label{caccioppoli}
\end{align}
with $C_1\geq 1$ depending only on $\lambda$, $q$, $L$, $a(M)$ and $M$. 
\end{lem}
\begin{proof}
Assume $x_0\in\Omega$ and $\rho\leq 1$ satisfy $B_{\rho}(x_0)\Subset\Omega$ and $\rho\leq\rho_1(\lvert\xi\rvert,\lvert\nu\rvert)$.
We denote $\xi+\nu(x-x_0)$ by $\ell(x)$ and take a standard cut-off function $\psi\in C^\infty_0(B_\rho(x_0))$
satisfying $0\leq\psi\leq 1$, $\lvert D\psi\rvert\leq 4/\rho$, $\psi\equiv 1$ on $B_{\rho/2}(x_0)$.
Then $\varphi :=\psi^q(u-\ell )$ is admissible as a test function in \eqref{ws}, and obtain
\begin{align}
 \dashint_{B_\rho(x_0)}\psi^q\langle &A(x,u,Du),Du-\nu \rangle dx \notag\\
 =-\,&\dashint_{B_\rho(x_0)}\langle A(x,u,Du),q\psi^{q-1}
 D\psi\otimes (u-\ell)\rangle dx
 +\dashint_{B_\rho(x_0)}\langle f,\varphi\rangle dx, \label{system2}
\end{align}
where $\xi\otimes\zeta:=\xi_i\zeta^\alpha$. We further have
\begin{align}
 -\,\dashint_{B_\rho(x_0)}\psi^q\langle &A(x,u,\nu),Du-\nu \rangle dx \notag\\
 =&\dashint_{B_\rho(x_0)}\langle A(x,u,\nu),q\psi^{q-1} \psi\otimes
 (u-\ell)\rangle dx-\dashint_{B_\rho(x_0)}\langle A(x,u,\nu),D\varphi
 \rangle dx, 
\end{align}
and
\begin{equation}
 \dashint_{B_\rho(x_0)}\langle A(x_0,\xi,\nu),D\varphi \rangle dx=0. \label{constcoeff}
\end{equation}
Adding these equations, from \eqref{system2} to \eqref{constcoeff}, we obtain 
\begin{align}
&\dashint_{B_\rho(x_0)}\psi^q \langle A(x,u,Du)-A(x,u,\nu), Du-\nu\rangle dx \notag \\
 =&-\dashint_{B_\rho(x_0)}\langle A(x,u,Du)-A(x,u,\nu),q\psi^{q-1} D\psi\otimes (u-\ell)\rangle dx \notag \\
 &-\dashint_{B_\rho(x_0)}\langle A(x,u,\nu)-A(x,\ell,\nu),D\varphi\rangle dx \notag \\
 &-\dashint_{B_\rho(x_0)}\langle A(x,\ell,\nu)-A(x_0,\xi,\nu) ,D\varphi \rangle dx \notag \\
 &+\dashint_{B_\rho(x_0)}\langle f,\varphi\rangle dx \notag \\
 =:& \>\> \I+\II+\III+\IV. \label{caccio-devide}
\end{align}
The terms I, II, III and IV are defined above. Using the ellipticity condition ({\bf H2}) to the left hand side 
of \eqref{caccio-devide}, we get
\begin{align*}
 &\langle A(x,u,Du)-A(x,u,\nu),Du-\nu\rangle \\
 =&\int_0^1\left\langle D_p A(x,u,sDu+(1-s)\nu)(Du-\nu),
 Du-\nu\right\rangle ds \\
 \geq& \lambda \lvert Du-\nu\rvert^2\int_0^1(1+\lvert sDu+(1-s)\nu\rvert)^{q-2}ds.
\end{align*}
Then we estimate above by \eqref{GM2} in Lemma \ref{GM} and obtain
\begin{align}
 &\langle A(x,u,Du)-A(x,u,\nu),Du-\nu\rangle \notag \\
 \geq& 2^{(12-9q)/2}\lambda
 \left\{(1+\lvert\nu\rvert)^{q-2}\lvert Du-\nu\rvert^2+\lvert Du-\nu\rvert^q\right\} .\label{elliptic}
\end{align}
For $\varepsilon >0$ to be fixed later, using ({\bf H1}) and Young's inequality, we have
\begin{align}
 \lvert\,\I\,\rvert 
 \leq &\varepsilon\dashint_{B_\rho(x_0)}\psi^q\left\{(1+\lvert\nu\rvert)^{q-2}
 \lvert Du-\nu\rvert^2+\lvert Du-\nu\rvert^q\right\}dx \notag\\
 &+c(p,L,\varepsilon) \dashint_{B_\rho(x_0)}\left\{ (1+\lvert\nu\rvert)^{q-2}\left\lvert\frac{u-\ell}{\rho}\right\rvert^2 
 +\left\lvert\frac{u-\ell}{\rho}\right\rvert^q\right\}dx. \label{caccio-I}
\end{align}
In order to estimate II, we first use ({\bf H3}) and $D\varphi=\psi^q(Du-\nu)+q\psi^{q-1}D\psi\otimes
(u-\ell)$, we get
\begin{align*}
 \lvert\,\II\,\rvert 
 \leq &\dashint_{B_\rho(x_0)}\kappa(\lvert\xi\rvert+\lvert\nu\rvert\rho) 
  \mu\left(\lvert u-\ell\rvert\right)(1+\lvert\nu\rvert)^{q-1} \psi^q\lvert Du-\nu\rvert dx \\
  &+\dashint_{B_\rho(x_0)}\kappa(\lvert\xi\rvert+\lvert\nu\rvert\rho)\mu\left(\lvert u-\ell\rvert\right)(1+\lvert\nu\rvert)^{q-1}
 q\psi^{q-1}\lvert D\psi\rvert\lvert u-\ell\rvert dx \\
 =:& \II_1+\II_2.
\end{align*}
The terms $\II_1$ and $\II_2$ are defined above. Using Young's inequality we estimate $\II_1$ as
\begin{align*}
 \lvert\,\II_1\rvert 
 \leq &\varepsilon\dashint_{B_\rho(x_0)}\psi^q(1+\lvert\nu\rvert)^{q-2}\lvert Du-\nu\rvert^2 dx
 +\frac{1}{\varepsilon}\dashint_{B_\rho(x_0)}(1+\lvert\nu\rvert)^q 
 \kappa^2(\lvert\xi\rvert+\lvert\nu\rvert)\eta\left(\lvert u-\ell\rvert^2\right) dx. 
\end{align*}
Note that our choice $\rho\leq\rho_1(\lvert\xi\rvert,\lvert\nu\rvert)$ allow us to apply ($\eta 5$), so that we get
\begin{align*}
 \lvert\,\II_1\rvert 
 \leq &\varepsilon\dashint_{B_\rho(x_0)}\psi^q(1+\lvert\nu\rvert)^{q-2}\lvert Du-\nu\rvert^2dx 
  +\frac{1}{\varepsilon}\dashint_{B_\rho(x_0)}(1+\lvert\nu\rvert)^{q-2}\left\lvert \frac{u-\ell}{\rho}\right\rvert^2dx \\
 &+\frac{1}{\varepsilon}\dashint_{B_\rho(x_0)}(1+\lvert\nu\rvert)^q \kappa^2(\lvert\xi\rvert+\lvert\nu\rvert) 
 \eta\Bigl(\rho^2(1+\lvert\nu\rvert)^2 \kappa^2(\lvert\xi\rvert+\lvert\nu\rvert)\Bigr)dx.  
\end{align*}
Using the definition of $G(\cdot,\cdot)$ and the fact that $\eta(ct)\leq c\eta(t)$ for $c\geq 1$, we deduce
\begin{align*}
 \lvert\,\II_1\rvert 
 \leq &\varepsilon\dashint_{B_\rho(x_0)}\psi^q(1+\lvert\nu\rvert)^{q-2}\lvert Du-\nu\rvert^2dx 
 +\frac{1}{\varepsilon}\dashint_{B_\rho(x_0)}(1+\lvert\nu\rvert)^{q-2}\left\lvert \frac{u-\ell}{\rho}\right\rvert^2dx \\
 &+\frac{1}{\varepsilon}(1+\lvert\nu\rvert)^qG(\lvert\xi\rvert,\lvert\nu\rvert)\eta\left(\rho^2\right).  
\end{align*}
Similarly we see
\begin{align*}
 \lvert\,\II_2\rvert 
 \leq &c(q,\varepsilon) 
 \dashint_{B_\rho(x_0)}(1+\lvert\nu\rvert)^{q-2}\left\lvert \frac{u-\ell}{\rho}\right\rvert^2dx 
 +c(q,\varepsilon)(1+\lvert\nu\rvert)^qG(\lvert\xi\rvert,\lvert\nu\rvert)\eta\left(\rho^2\right).
\end{align*}
Combining these two estimates and get
\begin{align}
 \lvert\,\II\,\rvert 
 \leq &\varepsilon\dashint_{B_\rho(x_0)}\psi^q(1+\lvert\nu\rvert)^{q-2}\lvert Du-\nu\rvert^2dx 
 +c(q,\varepsilon)\dashint_{B_\rho(x_0)}(1+\lvert\nu\rvert)^{q-2}\left\lvert\frac{u-\ell}{\rho}\right\rvert^qdx \notag \\
 &+c(q,\varepsilon)(1+\lvert\nu\rvert)^qG(\lvert\xi\rvert,\lvert\nu\rvert)\eta\left(\rho^2\right) . \label{caccio-II}
\end{align}
In the same way we derive 
\begin{align}
 \lvert\,\III\,\rvert 
 \leq & \dashint_{B_\rho(x_0)}(1+\lvert\nu\rvert)^{q-1}\kappa(\lvert\xi\rvert+\lvert\nu\rvert) 
 \mu\left( (1+\lvert\nu\rvert)\rho\right)\psi^q\lvert Du-\nu\rvert dx \notag\\
 &+ \dashint_{B_\rho(x_0)}(1+\lvert\nu\rvert)^{q-1}\kappa(\lvert\xi\rvert+\lvert\nu\rvert\rho) 
 \mu\bigl( (1+\lvert\nu\rvert)\rho\bigr) 4q\left\lvert\frac{u-\ell}{\rho}\right\rvert dx \notag\\
 \leq & \varepsilon\dashint_{B_\rho(x_0)}\psi^q(1+\lvert\nu\rvert)^{q-2}\lvert Du-\nu\rvert^2dx
 +\varepsilon\dashint_{B_\rho(x_0)}(1+\lvert\nu\rvert)^{q-2}\left\lvert\frac{u-\ell}{\rho}\right\rvert^2dx \notag\\
 &+c(q,\varepsilon)(1+\lvert\nu\rvert)^qG(\lvert\xi\rvert,\lvert\nu\rvert)\eta(\rho^2). \label{caccio-III}
\end{align}
For $\varepsilon'>0$ to be fixed later, using ({\bf H4}), Lemma \ref{Young2}, and Young's inequality, we have
\begin{align}
\lvert\,\IV\,\rvert 
 \leq & \dashint_{B_\rho(x_0)}(a\lvert Du\rvert^q+b)\psi^q\lvert u-\ell\rvert dx \notag\\
 \leq & a(1+\varepsilon')\dashint_{B_\rho(x_0)}\psi^q\lvert Du-\nu\rvert^qlvert u-\ell\rvert dx
 +\varepsilon b^2\rho^2 +\frac{1}{\varepsilon}\dashint_{B_\rho(x_0)}\left\lvert\frac{u-\ell}{\rho}\right\rvert^2dx \notag\\
 &+\dashint_{B_\rho(x_0)}\left\{ aK(q,\varepsilon')\rho\lvert\nu\rvert^{(q+2)/2}\right\}
 (1+\lvert\nu\rvert)^{(q-2)/2}\left\lvert\frac{u-\ell}{\rho}\right\rvert dx \notag\\
 \leq & a(1+\varepsilon^\prime)(2M+\lvert\nu\rvert\rho)
 \dashint_{B_\rho(x_0)}\psi^q\lvert Du-\nu\rvert^qdx +\frac{2}{\varepsilon}\dashint_{B_\rho(x_0)}
 (1+\lvert\nu\rvert)^{q-2}\left\lvert\frac{u-\ell}{\rho}\right\rvert^2dx \notag\\
 &+\varepsilon(1+\lvert\nu\rvert)^q\rho^2\left\{ aK(q,\varepsilon')\lvert\nu\rvert +b\right\}^2. \label{caccio-IV}
\end{align}
Combining above estimates, from \eqref{caccio-devide} to \eqref{caccio-IV}, and set
$\lambda'=2^{(12-9q)/2}\lambda$C
$\Lambda :=\lambda'-3\varepsilon-a(1+\varepsilon')(2M+\lvert\nu\rvert\rho)$, this gives
\begin{align}
 &\Lambda\dashint_{B_\rho(x_0)}\psi^q \left\{(1+\lvert\nu\rvert)^{q-2}\lvert Du-\nu\rvert^2+\lvert Du-\nu\rvert^q\right\}dx \notag\\
 \leq & c(q,L,\varepsilon) 
 \left[\dashint_{B_\rho(x_0)} \left\{(1+\lvert\nu\rvert)^{q-2}\left\lvert\frac{u-\ell}{\rho}\right\rvert^2
 +\left\lvert\frac{u-\ell}{\rho}\right\rvert^q\right\}dx 
 +(1+\lvert\nu\rvert)^qG(\lvert\xi\rvert,\lvert\nu\rvert)\eta(\rho^2) \right] \notag\\
 &+\varepsilon(1+\lvert\nu\rvert)^q\left\{ aK\lvert\nu\rvert +b\right\}^2\rho^2. \label{roughcaccio}
\end{align}
Now choose $\varepsilon=\varepsilon(\lambda ,p,a(M),M)>0$ and $\varepsilon'=\varepsilon'(\lambda ,p,a(M),M)>0$ in a right way 
(for more precise way of choosing $\varepsilon$ and $\varepsilon'$, we refer to \cite[Lemma 4.1]{DG}), 
we obtain \eqref{caccioppoli}. 
\end{proof}

\section{Approximatively $\A$-harmonic functions}
\begin{lem}\label{A-harm2}
Under the same assumption in Lemma \ref{caccioppoli}, take $\xi=u_{x_0,\rho}$. Then for any $x_0\in\Omega$ 
and $\rho\leq \rho_1(\lvert\xi\rvert,\lvert\nu\rvert)$ satisfy $B_\rho(x_0)\Subset\Omega$, the inequality
\begin{align}
 \dashint_{B_\rho(x_0)}\A (Dv,D\varphi )dx \leq C_2(1+\lvert\nu\rvert)&\Biggl[ 
 \omega^{1/2}\left(lvert\xi\rvert+\lvert\nu\rvert,\Phi(x_0,\rho,\nu)\right)\Phi^{1/2}(x_0,\rho,\nu) \notag\\
 &+\Phi(x_0,\rho,\nu) +G(\lvert\xi\rvert,\lvert\nu\rvert)\sqrt{\eta(\rho^2)} 
 +\rho(a\lvert\nu\rvert +b)\Biggr]\sup_{B_\rho(x_0)}\lvert D\varphi\rvert \label{Ah}
\end{align}
holds for all $\varphi\in C^\infty_0(B_\rho(x_0),\R^N)$. Where 
\begin{align*}
 v :=& u-\ell = u-\xi-\nu(x-x_0), \\
 \A (Dv,D\varphi ):=& \frac{1}{(1+\lvert\nu\rvert)^{q-1}}\left\langle D_pA(x_0,\xi,\nu )Dv,D\varphi\right\rangle, \\
 \Phi(x_0,\rho,\nu):=& \dashint_{B_\rho(x_0)}\left\{\frac{\lvert Du-\nu\rvert^2}{(1+\lvert\nu\rvert)^2} 
  +\frac{\lvert Du-\nu\rvert^q}{(1+\lvert\nu\rvert)^q}\right\} dx 
\end{align*}
and $C_2\geq 1$ depending only on $n$, $q$, $L$ and $a(M)$. 
\end{lem}
\begin{proof}
Assume $x_0\in\Omega$ and $\rho\leq 1$ which satisfies $B_{\rho}(x_0)\Subset\Omega$ and 
$\rho\leq\rho_1(\lvert\xi\rvert,\lvert\nu\rvert)$.
Without loss of generality we may assume $\displaystyle\sup_{B_\rho(x_0)}\lvert D\varphi\rvert\leq 1$.
Note that this implies $\displaystyle\sup_{B_\rho(x_0)}\lvert\varphi\rvert\leq\rho\leq 1$. Using the fact that 
$\int_{B_\rho(x_0)}A(x_0,\xi,\nu)D\varphi dx=0$ holds for all $\varphi\in C^\infty_0(B_\rho(x_0),\R^N)$ we deduce
\begin{align}
 (1+\lvert\nu\rvert)^{q-1}&\dashint_{B_\rho(x_0)}\mathcal{A}(Dv,D\varphi)dx \notag\\
 =&\dashint_{B_\rho(x_0)}\int_0^1\langle\left[ D_p A(x_0,\xi,\nu)- D_pA(x_0,\xi,\nu+s(Du-\nu))
 \right] (Du-\nu),D\varphi\rangle dsdx \notag\\
 &+\dashint_{B_\rho(x_0)}\langle A(x_0,\xi,Du) -A(x,\ell,Du),D\varphi\rangle dx \notag\\
 &+\dashint_{B_\rho(x_0)}\langle A(x,\ell,Du)-A(x,u,Du),D\varphi\rangle dx \notag\\
 &+\dashint_{B_\rho(x_0)}\langle f,\varphi\rangle dx \notag\\
 =&:\I+\II+\III+\IV \label{Ah-devide}
\end{align}
where terms I, II, III and IV are define above. 

We estimate I using the modulus of continuity $\omega(\cdot,\cdot)$ from ({\bf H1}), the Jensen's inequality 
and H\"{o}lder's inequality, and we get
\begin{align}
 \lvert\,\I\,\rvert 
 &\leq c(q,L)\dashint_{B_\rho(x_0)}\int_0^1 \omega\left(\lvert\xi\rvert+\lvert\nu\rvert,\lvert Du-\nu\rvert^2\right)
 (1+\lvert\nu\rvert+\lvert Du-\nu\rvert)^{q-2}\lvert Du-\nu\rvert dsdx \notag\\
 &\leq 
 c\, (1+\lvert\nu\rvert)^{q-1} \dashint_{B_\rho(x_0)}\omega\left(\lvert\xi\rvert+\lvert\nu\rvert,\lvert Du-D\ell\rvert^2\right)
 \left\{\frac{\lvert Du-\nu\rvert}{1+\lvert\nu\rvert} +\frac{\lvert Du-\nu\rvert^{q-1}}{(1+\lvert\nu\rvert)^{q-1}}\right\}dx \notag\\
 &\leq c\, (1+\lvert\nu\rvert)^{q-1}\left[ \omega^{1/2}\left(\lvert\xi\rvert+\lvert\nu\rvert, 
 (1+\lvert\nu\rvert)^2\Phi(x_0,\rho,\nu)\right)\Phi^{1/2}(x_0,\rho,\nu)\right. \notag\\
 &\qquad\qquad\qquad\quad \left. +\omega^{1/q}\left(\lvert\xi\rvert+\lvert\nu\rvert,(1+\lvert\nu\rvert)^2\Phi(x_0,\rho,\nu)\right)
 \Phi^{1/{q_*}}(x_0,\rho,\ell)\right]\notag\\
 &\leq c\, (1+\lvert\nu\rvert)^q\left[ \omega^{1/2}\left(\lvert\xi\rvert+\lvert\nu\rvert,\Phi(x_0,\rho,\nu)\right)
 \Phi^{1/2}(x_0,\rho,\nu)+\Phi(x_0,\rho,\nu)\right] , \label{Ah-I}
\end{align}
where $q_*>0$ is the dual exponent of $q\geq 2$, i.e., $q_*=q/(q-1)$. 
The last inequality following from the fact that $a^{1/q}b^{1/q_*}=a^{1/q}b^{1/q}b^{(q-2)/q}\leq
a^{1/2}b^{1/2}+b$ holds by Young's inequality and the fact that $\omega(s,ct)\leq c\omega(s,t)$ for $c\geq 1$ which deduce from 
the concavity of $t\mapsto \omega(s,t)$. 

In the same way, using the modulus of continuity $\eta(\cdot)$ from ({\bf H3}), Young's inequality and, we deduce
\begin{align}
 \lvert\,\II\,\rvert 
 \leq &2^{q-2}\kappa(\lvert\xi\rvert+\lvert\nu\rvert)(1+\lvert\nu\rvert)^q \sqrt{\eta(\rho^2)} \notag\\
 &+2^{q-2}\dashint_{B_\rho(x_0)}\kappa(\lvert\xi\rvert+\lvert\nu\rvert) 
 \sqrt{\eta(\rho^2(1+\lvert\nu\rvert)^2)} \lvert Du-\nu\rvert^{q-1}dx \notag\\
 \leq &2^{q-1}(1+\lvert\nu\rvert)^qG(\lvert\xi\rvert,\lvert\nu\rvert)\sqrt{\eta(\rho^2)} 
 +2^{q-2}(1+\lvert\nu\rvert)^q\Phi(x_0,\rho,\nu) .
\end{align}
Here we have used $\eta^{q/2}(\rho^2(1+\lvert\nu\rvert)^2)\leq \sqrt{\eta(\rho^2(1+\lvert\nu\rvert)^2)}$ which follows from 
the nondecreasing property of $t\mapsto \eta(t)$, ($\eta$4) and our assumption $\rho\leq\rho_1\leq 1$. 

We derive, using again the modulus of continuity $\eta(\cdot)$ from ({\bf H3}),
\begin{align*}
 \lvert\,\III\,\rvert 
 \leq & c(q)\dashint_{B_\rho(x_0)}\kappa(\lvert\xi\rvert+\lvert\nu\rvert)\sqrt{\eta(\lvert u-\ell\rvert^2)}(1+\lvert\nu\rvert)^{q-1}dx \\
 &+c(q)\dashint_{B_\rho(x_0)}\kappa(\lvert\xi\rvert+\lvert\nu\rvert)\sqrt{\eta(\lvert u-\ell\rvert^2)}\lvert Du-\nu\rvert^{q-1}dx \\
 =: &\III_1+\III_2,
\end{align*}
where the terms $\III_1$ and $\III_2$ are defined above.
Using H\"{o}lder's inequality, Jensen's inequality, ($\eta$6) and the Poincar\'{e} inequality, we have
\begin{align*}
 \III_1
 &\leq c(q) (1+\lvert\nu\rvert)^{q-1}\kappa(\lvert\xi\rvert+\lvert\nu\rvert) 
 \eta^{1/2}\left(\dashint_{B_\rho(x_0)}\lvert u-\ell\rvert^2 dx\right) \\
 &\leq c\, \rho^{-2}(1+\lvert\nu\rvert)^{q-2}\left\{\rho^2(1+\lvert\nu\rvert)^2 
 \kappa^2(\lvert\xi\rvert+\lvert\nu\rvert)\eta^{1/2}\Bigl(\rho^2(1+\lvert\nu\rvert)^2\kappa^2(\lvert\xi\rvert+\lvert\nu\rvert)\Bigr) 
 +\dashint_{B_\rho(x_0)}\lvert u-\ell\rvert^2 dx\right\} \\ 
 &\leq c(q)(1+\lvert\nu\rvert)^qG(\lvert\xi\rvert,\lvert\nu\rvert)\sqrt{\eta\left(\rho^2\right)} 
 +c(n,q)(1+\lvert\nu\rvert)^q\Phi(x_0,\rho,\nu) .
\end{align*}
Similarly, we have, using Young's inequality, ($\eta$5) and the Poincar\'{e} inequality,
\begin{align*}
 \III_2
 \leq &c(q)\dashint_{B_\rho(x_0)}\kappa^q(\lvert\xi\rvert+\lvert\nu\rvert)\eta^{q/2}\left(\lvert u-\ell\rvert^2\right) dx
  +c(q)\dashint_{B_\rho(x_0)}\lvert Du-\nu\rvert^qdx \\
 \leq &c\, \dashint_{B_\rho(x_0)}\Bigl[\rho^{-2}\left\{ 
  \kappa^2(\lvert\xi\rvert+\lvert\nu\rvert)\rho^2\eta\left(\kappa^2(\lvert\xi\rvert+\lvert\nu\rvert)\rho^2\right) 
  +\lvert u-\ell\rvert^2\right\} \Bigr]^{q/2} dx
  +c\, (1+\lvert\nu\rvert)^q\Phi(x_0,\rho,\nu) \\
 \leq &c(1+\lvert\nu\rvert)^qG(\lvert\xi\rvert,\lvert\nu\rvert)\sqrt{\eta\left(\rho^2\right)}
  +c(n,q)(1+\lvert\nu\rvert)^q\Phi(x_0,\rho,\nu) .
\end{align*}
Thus we obtain
\begin{align}
 \lvert\,\III\,\rvert 
 \leq &c(q)(1+\lvert\nu\rvert)^qG(\lvert\xi\rvert,\lvert\nu\rvert)\sqrt{\eta\left(\rho^2\right)} 
 +c(n,q)(1+\lvert\nu\rvert)^q\Phi(x_0,\rho,\nu). 
\end{align} 
Using ({\bf H4}) and recall that $\displaystyle\sup_{B_\rho(x_0)}\lvert\varphi\rvert\leq\rho$ holds, we have
\begin{align}
 \lvert\,\IV\,\rvert
 &\leq \dashint_{B_\rho(x_0)}\rho a(\lvert Du-\nu\rvert+\lvert\nu\rvert)^qdx +b\rho \notag\\
 &\leq 2^{q-1}a(1+\lvert\nu\rvert)^q\Phi(x_0,\rho,\nu)
  +2^{q-1}\rho(1+\lvert\nu\rvert)^q(a\lvert\nu\rvert +b). \label{25}
\end{align}
Combining these estimates, from \eqref{Ah-devide} to (\ref{25}), we obtain the conclusion.
\end{proof}
\section{Proof of the Regularity Theorem}
Let write $\Phi(\rho) =\Phi(x_0,\rho,(Du)_{x_0,\rho})$ from now on.
Now we are in the position to establish the excess improvement.
\begin{lem}\label{EI}
Assume the same assumption with Lemma \ref{A-harm2}. Let $\theta\in (0,1/8]$ be arbitrary and impose the 
following smallness conditions on the excess:
\begin{enumerate}
 \item[{\rm (i)}] $\displaystyle{\omega^{1/2}\left(\lvert u_{x_0,\rho}\rvert+\lvert (Du)_{x_0,\rho}\rvert, \Phi(\rho)\right)
 +\sqrt{\Phi(\rho)}\leq\frac{\delta}{2}}$ with the constant 
 $\delta =\delta(n,N,q,\lambda,L,\theta^{n+q+2})$ from Lemma \ref{A-harm};
 \item[{\rm (ii)}] $(1+\lvert(Du)_{x_0,\rho}\rvert)\gamma(\rho)\leq \theta^n\left( 2\sqrt{C_0\tilde{C}}\right)^{-1}$, where \\
 $C_0$ and $\tilde{C}$ are constants from Theorem \ref{Campanato} and Lemma \ref{A-harm}, and \\
 $\displaystyle{\gamma(\rho):=C_2\left[\sqrt{\Phi(\rho)}+2\delta^{-1}
 \left\{ G(\lvert u_{x_0,\rho}\rvert,\lvert(Du)_{x_0,\rho}\rvert)\sqrt{\eta(\rho^2)}+\rho(a(1+\lvert(Du)_{x_0,\rho}\rvert)+b)
 \right\}\right]}$.
 \item[{\rm (iii)}] $\rho\leq\rho_1(\lvert u_{x_0,\rho}\rvert,\lvert(Du)_{x_0,\rho}\rvert)$.
\end{enumerate}
Then there holds the excess improvement estimate
\begin{equation}
 \Phi(\theta\rho)\leq C_3\theta^2\Phi(\rho)+H(\lvert u_{x_0,\rho}\rvert,\lvert(Du)_{x_0,\rho}\rvert)\eta(\rho^2),
\end{equation}
with a constant $C_3$ that depends only on $n$, $N$, $\lambda$, $L$, $q$, $a(M)$ and $M$. Here $H(\cdot,\cdot)$ is defined as
\[
 H(s,t):=8\delta^{-2}C_3\left[ G^2(1+s,1+t)+\{a(1+t)+b\}\right].
\]
\end{lem}
\begin{proof}
We consider $B_\rho(x_0)\Subset\Omega$ and set $\xi=u_{x_0,\rho}, \nu=(Du)_{x_0,\rho},\ell(x)=\xi+\nu(x-x_0)$.
Assume (i), (ii) and (iii) are satisfied and we rescale $u$ as 
\[
 w:=\frac{u-\ell}{(1+\lvert\nu\rvert)\gamma}.
\]
Applying Lemma \ref{A-harm2} on $B_\rho(x_0)$ to $w$ and combining (i), we obtain
\begin{align*}
 &\dashint_{B_\rho(x_0)}\A(Dw,D\varphi)dx \\
 \leq & \left[\omega^{1/2}\left(\lvert\xi\rvert+\lvert\nu\rvert,\sqrt{\Phi(\rho)}\right)
 +\sqrt{\Phi(\rho)}+\frac{\delta}{2} \right]
 \sup_{B_{\rho}(x_0)}\lvert D\varphi\rvert \\
 \leq& \delta\sup_{B_{\rho}(x_0)}\lvert D\varphi\rvert
\end{align*}
for all $\varphi\in C_0^\infty(B_\rho(x_0),\R^N)$.
Moreover, we have, note that $\gamma\geq C_2\sqrt{\Phi(\rho)}$ holds from the definition of $\gamma$,
\begin{align*}
 \dashint_{B_{\rho}(x_0)}\left\{\lvert Dw\rvert^2+\gamma^{q-2}\lvert Dw\rvert^p\right\}dx
 =&\dashint_{B_{\rho}(x_0)}\left\{ \frac{\lvert Du-\nu\rvert^2}{\gamma^2(1+\lvert\nu\rvert)^2}
 +\gamma^{q-2}\frac{\lvert Du-\nu\rvert^q}{\gamma^q(1+\lvert\nu\rvert)^q} \right\} dx \\
 \leq & \frac{\Phi(\rho)}{\gamma^2}
 \leq \frac{1}{{C_2}^2}
 \leq 1.
\end{align*}
Thus, these two inequalities allow us to apply the $\A$-harmonic approximation lemma (Lemma \ref{A-harm}), to
conclude the existence of an $\A$-harmonic function $h$ satisfying
\begin{align}
 \dashint_{B_{\rho/2}(x_0)}&\left\{\left\lvert\frac{w-h}{\rho/2}\right\rvert^2
 +\gamma^{q-2}\left\lvert\frac{w-h}{\rho/2}\right\rvert^q\right\} dx
 \leq\theta^{n+q+2}, \quad \text{and} \\
 \dashint_{B_{\rho/2}(x_0)}&\left\{\lvert Dh\rvert^2
 +\gamma^{q-2}\lvert Dh\rvert^q\right\}dx\leq \tilde{C}, \label{energybound}
\end{align}
where we taken $\varepsilon=\theta^{n+q+2}$. From Theorem \ref{Campanato} and \eqref{energybound} we have
\[
 \sup_{B_{\rho/4}(x_0)}\lvert D^2h\rvert^2\leq 4C_0\tilde{C}\rho^{-2}.
\]
From this we infer the following estimate for $s=2$ as well as for $s=q$,
\[
 \sup_{B_{\rho/4}(x_0)}\lvert D^2h\rvert^s
 \leq c(n,N,\lambda, L, q,s)\rho^{-s}.
\]
For $\theta\in(0,1/8]$, Taylor's theorem applied to $h$ at $x_0$ yields
\[
 \sup_{x\in B_{2\theta\rho}(x_0)}\lvert h(x)-h(x_0)-Dh(x_0)(x-x_0)\rvert^s
 \leq c(n,N,\lambda, L, q,s)\theta^{2s}\rho^s.
\]
We have then
\begin{align*}
 &\gamma^{s-2}(2\theta\rho)^{-s}\dashint_{B_{2\theta\rho}(x_0)}
 \lvert w-h(x_0)-Dh(x_0)(x-x_0)\rvert^sdx \\
 \leq & c(s)\gamma^{s-2}(2\theta\rho)^{-s}\left[\dashint_{B_{2\theta\rho}(x_0)}\lvert w-h\rvert^sdx
 +\dashint_{B_{2\theta\rho}(x_0)}\lvert h(x)-h(x_0)-Dh(x_0)(x-x_0)\rvert^sdx\right]\\
 \leq & c(n,N,\lambda,L,q,s)\theta^2.
\end{align*}
Recall that the mean-value of $u-(\nu+\gamma(1+\lvert\nu\rvert)Dh(x_0))(x-x_0)$ on $B_{2\theta\rho}(x_0)$ is $u_{x_0,2\theta\rho}$, 
we have
\begin{align}
 &(2\theta\rho)^{-s}\dashint_{B_{2\theta\rho}(x_0)}\lvert u-u_{x_0,2\theta\rho}
 -\left(\nu+\gamma(1+\lvert\nu\rvert)Dh(x_0)\right)(x-x_0)\rvert^sdx \notag\\
 \leq&
 c(s)(2\theta\rho)^{-s}\gamma^s(1+\lvert\nu\rvert)^s
 \dashint_{B_{2\theta\rho}(x_0)}\lvert w-h(x_0)-Dh(x_0)(x-x_0)\rvert^sdx \notag\\
 \leq & c(n,N,\lambda,L,q,s)(1+\lvert\nu\rvert)^s\theta^2\gamma^2. \label{scaling}
\end{align}
By assumption (ii), we infer $\sqrt{\Phi(\rho)}\leq \theta^n/2$. This yields
\begin{equation}
 \lvert(Du)_{x_0,\theta\rho}-\nu\rvert
 \leq\theta^{-n}\dashint_{B_{\rho}(x_0)}\lvert Du-\nu\rvert dx 
 \leq \theta^{-n}(1+\lvert\nu\rvert)\sqrt{\Phi(\rho)} 
 \leq \frac{1}{2}(1+\lvert\nu\rvert).
\end{equation}
Thus, combining with the estimate $1+\lvert\nu\rvert\leq1+\lvert(Du)_{x_0,\theta\rho}\rvert
+\lvert(Du)_{x_0,\theta\rho}-\nu\rvert$, we obtain
\begin{equation}
 1+\lvert\nu\rvert\leq 2(1+\lvert(Du)_{x_0,\theta\rho}\rvert). \label{nuesti}
\end{equation}
Set $P_0=\nu+\gamma(1+\lvert\nu\rvert)Dh(x_0)$. Then Theorem \ref{Campanato}, \eqref{energybound} and assumption (ii) imply
\begin{align} 
 \lvert P_0\rvert\leq \lvert\nu\rvert+\lvert\gamma(1+\lvert\nu\rvert)Dh(x_0)\rvert 
 \leq \lvert\nu\rvert+\gamma(1+\lvert\nu\rvert)\sqrt{C_0c(n,q)} 
 \leq \frac{1}{2}+\lvert\nu\rvert. \label{P-esti}
\end{align}
Therefore, combining with \eqref{nuesti}, we have
\begin{equation}
 1+\lvert P_0\rvert\leq 3(1+\lvert(Du)_{x_0,\theta\rho}\rvert).
\end{equation}
Applying Caccioppoli-type inequality (Lemma \ref{caccioppoli}) on $B_{2\theta\rho}(x_0)$ with $\xi=u_{x_0,2\theta\rho}$
and $\nu=P_0$ yields
\begin{align}
 \Phi(\theta\rho) \leq &6^q\Phi(x_0,\theta\rho,P_0) \notag\\
 \leq & 6^qC_1\Bigg[ \dashint_{B_{2\theta\rho}(x_0)}\left\{
 \left\lvert\frac{u-u_{x_0,2\theta\rho}-P_0(x-x_0)}{2\theta\rho(1+\lvert P_0\rvert)}\right\rvert^2+
 \left\lvert\frac{u-u_{x_0,2\theta\rho}-P_0(x-x_0)}{2\theta\rho(1+\lvert P_0\rvert)}\right\rvert^q\right\} dx \notag\\
 &\qquad +G(\lvert u_{x_0,2\theta\rho}\rvert,\lvert P_0\rvert)\eta((2\theta\rho)^2)
 +\left( a\lvert P_0\rvert+b\right)^2 (2\theta\rho)^2 \Bigg]. \label{appliedcaccioppoli}
\end{align}
Using H\"{o}lder's inequality, the Poincar\'{e} inequality and assumption (ii) we have
\begin{align}
 \lvert u_{x_0,2\theta\rho}\rvert \leq &\lvert u_{x_0,\rho}\rvert+\left\lvert\dashint_{B_{2\theta\rho}(x_0)}
  (u-u_{x_0,\rho}-\nu(x-x_0))dx\right\rvert \notag\\
 \leq &\lvert u_{x_0,\rho}\rvert+\left(\dashint_{B_{2\theta\rho}(x_0)}
  \lvert u-u_{x_0,\rho}-\nu(x-x_0)\rvert^2dx\right)^{1/2} \notag\\
 \leq &\lvert u_{x_0,\rho}\rvert+(2\theta)^{-n/2}\left(\dashint_{B_\rho(x_0)}
  \lvert u-u_{x_0,\rho}-\nu(x-x_0)\rvert^2dx\right)^{1/2} \notag\\
 \leq &\lvert u_{x_0,\rho}\rvert+\theta^{-n/2}\sqrt{C_P}(1+\lvert\nu\rvert)\sqrt{\Phi(\rho)} \notag\\
 \leq &\lvert u_{x_0,\rho}\rvert+\theta^{-n/2}\frac{\sqrt{C_P}}{C_2}(1+\lvert\nu\rvert)\gamma \notag\\
 \leq &\lvert u_{x_0,\rho}\rvert+1.
\end{align}
Set $H_0(s,t):=G^2(1+s,1+t)+\{ a(1+t)+b\}^{q_*}$ and using \eqref{P-esti} we obtain
\begin{equation}
 G(\lvert u_{x_0,2\theta\rho}\rvert,\lvert P_0\rvert)\eta((2\theta\rho)^2)
 +\left( a\lvert P_0\rvert+b\right)^2 (2\theta\rho)^2 
 \leq H_0(\lvert\xi\rvert,\lvert\nu\rvert)\eta(\rho^2). \label{Hesti}
\end{equation}
The definition of $\gamma$ and $H_0$ imply
\begin{align}
 \gamma^2 \leq &2{C_2}^2\left[\Phi(\rho)+4\delta^{-2}
 \left\{ G(\lvert\xi\rvert,\lvert\nu\rvert)\sqrt{\eta(\rho^2)}+\rho(a(1+\lvert\nu\rvert)+b)\right\}^2\right] \notag\\
 \leq &2{C_2}^2\left[\Phi(\rho)+ 8\delta^{-2}H_0(\lvert\xi\rvert,\lvert\nu\rvert)\eta(\rho^2)\right] . \label{gammaesti}
\end{align}
Plugging \eqref{scaling}, \eqref{Hesti}, \eqref{gammaesti} into \eqref{appliedcaccioppoli}, we deduce 
\begin{align*}
 \Phi(\theta\rho) \leq &6^qC_1\bigg[ c(n,N,\lambda,L,q)\theta^2\gamma^2 
 +G(\lvert u_{x_0,2\theta\rho}\rvert,\lvert P_0\rvert)\eta((2\theta\rho)^2)
 +\left( a\lvert P_0\rvert+b\right)^2 (2\theta\rho)^2 \bigg] \\
 \leq &6^qC_1\bigg[ c\,\theta^2{C_2}^2\{\Phi(\rho)+\delta^{-2}H_0(\lvert\xi\rvert,\lvert\nu\rvert)\eta(\rho^2)\} 
 +H_0(\lvert\xi\rvert,\lvert\nu\rvert)\eta(\rho^{q_*}) \bigg] \notag\\
 \leq & C_3\bigg[ \theta^2\Phi(\rho)+8\delta^{-2}H_0(\lvert\xi\rvert,\lvert\nu\rvert)\eta(\rho^2)\bigg],
\end{align*}
and this complete the proof.
\end{proof}
For $\sigma\in[\alpha,1)$ we find $\theta\in(0,1/8]$ such that $C_3\theta^2\leq \theta^{2\sigma}/2$.
For $T_0 \geq 1$ there exists $\Phi_0>0$ such that
\begin{align}
 &\omega^{1/2}\left( 2T_0,\sqrt{2\Phi_0}\right)+\sqrt{2\Phi_0}\leq \frac{\delta}{2}, \label{sc1}\\
 &2C_4(1+2T_0)\sqrt{2\Phi_0}\leq \theta^n, \label{sc2}
\end{align}
where $C_4:=C_3\left( 1+\sqrt{C_P}\>\right)$. Note that $\Phi_0<1$. Then choose $0<\rho_0\leq 1$ such that
\begin{align}
 &C_5\sqrt{\eta(\rho_0)}\leq\Phi_0, \label{sc3} \\
 &\frac{(1+2T_0)(1+\sqrt{C_P})}{\theta^{n/2}}
  \sqrt{\frac{C_5\alpha^2\beta^2 \widetilde{F}({\rho_0}^2)}{4(1-\theta^{\alpha\beta})^2}}
  \leq \frac{1}{2}T_0, \label{sc4}
\end{align}
where 
\[
 C_5=C_5(n,N,\lambda,L,q,a,M,\alpha,\sigma,T_0):=\frac{2H(2T_0,2T_0)}{2\theta^{2\alpha}-\theta^{2\sigma}}.
\]
\begin{lem}\label{Iteration}
Assume that for some $T_0\geq 1$ and $B_\rho(x_0)\Subset\Omega$ we have
\begin{enumerate}
\item[(a)] $\lvert u_{x_0,\rho}\rvert+\lvert(Du)_{x_0,\rho}\rvert\leq T_0$,
\item[(b)] $\Phi(\rho)\leq \Phi_0$,
\item[(c)] $\rho\leq \rho_0$.
\end{enumerate}
Then the smallness conditions {\rm (i), (ii)} and {\rm (iii)} are satisfied on $B_{\theta^k\rho}(x_0)$ for $k\in\N\cup\{ 0\}$ in
Lemma \ref{EI}. Moreover, the limit
\[
 \Lambda_{x_0}:= \lim_{k\to\infty}(Du)_{x_0,\theta^k\rho}
\]
exists, and the inequality
\begin{equation}
 \dashint_{B_r(x_0)}\lvert Du-\Lambda_{x_0}\rvert^2dx \leq C_6
  \left[ \left(\frac{r}{\rho}\right)^{2\sigma}\Phi(\rho)+\widetilde{F}(r^2)\right] \label{campanatosemi}
\end{equation}
is valid for $0<r\leq \rho$ with a constant $C_6=C_6(n,N,\lambda,L,q,a(M),M,\alpha,\beta,\sigma,T_0)$.
\end{lem}
\begin{proof}
Inductively we shall derive for $k\in\N\cup\{ 0\}$ the following three assertions:
\begin{enumerate}
\item[($\I_k$)] $\Phi(\theta^k\rho)\leq 2\Phi_0$,
\item[($\II_k$)] $\lvert u_{x_0,\theta^k\rho}\rvert+\lvert(Du)_{x_0,\theta^k\rho}\rvert\leq 2T_0$,
\item[($\III_k$)] $\theta^k\rho\leq \rho_1(\lvert u_{x_0,\theta^k\rho}\rvert,\lvert(Du)_{x_0,\theta^k\rho}\rvert)$.
\end{enumerate}
We first note that ($\I_k$), ($\II_k$) and \eqref{sc1} imply the smallness condition $(\text{i}_k)$, i.e. (i) with 
$\theta^k\rho$ instead of $\rho$. Next we observe that ($\I_k$), ($\II_k$), \eqref{sc2} and \eqref{sc3} yield
\begin{align*}
 &(1+\lvert(Du)_{x_0,\theta^k\rho}\rvert)\left(2\sqrt{C_0\tilde{C}}\right)\gamma(\theta^k\rho) \\
 \leq &(1+\lvert(Du)_{x_0,\theta^k\rho}\rvert)\left[C_3\sqrt{2\Phi_0} 
  +H(\lvert u_{x_0,\theta^k\rho}\rvert,\lvert(Du)_{x_0,\theta^k\rho}\rvert)\sqrt{\eta({\rho_0}^2)}\right] \\
 \leq &(1+2T_0)\left[C_3\sqrt{2\Phi_0}+H(2T_0,2T_0)\sqrt{\eta({\rho_0}^2)}\right] \\
 \leq &(1+2T_0)\left[C_3\sqrt{2\Phi_0}+\frac{2\theta^{2\alpha}-\theta^{2\sigma}}{2}\Phi_0\right] \\
 \leq &2C_3(1+2T_0)\sqrt{2\Phi_0} \\
 \leq &1.
\end{align*}
Thus we have $(\text{ii}_k)$. Note that $C_2\left(2\sqrt{C_0\tilde{C}}\right)\leq C_3$ and $\Phi_0> 1$ are hold
from there definition. Finally $(\text{iii}_k)$ is just ($\III_k$). 

By (a), (b) and (c), there holds ($\I_0$),($\II_0$) and ($\III_0$).
Now suppose that we have ($\I_\ell$),($\II_\ell$) and ($\III_\ell$) for $\ell=0,1,\cdots,k-1$ with some $k\in\N$.
Then we can use Lemma \ref{EI} with $\rho,\theta\rho,\cdots,\theta^{k-1}\rho$, and yields
\begin{align*}
 \Phi(\theta^k\rho) 
 \leq &\left(\frac{1}{2}\theta^{2\sigma}\right)^k\Phi(\rho)+\sum_{\ell=0}^{k-1}\left(\frac{1}{2}\theta^{2\sigma}\right)^\ell
 H(\lvert u_{x_0,\theta^{k-1-\ell}\rho}\rvert,\lvert(Du)_{x_0,\theta^{k-1-\ell}}\rvert)\eta((\theta^{k-1-\ell}\rho)^2) \\
 \leq &\left(\frac{1}{2}\theta^{2\sigma}\right)^k\Phi(\rho)+H(2T_0,2T_0)\sum_{\ell=0}^{k-1}
 \left(\frac{1}{2}\theta^{2\sigma}\right)^\ell\eta((\theta^{k-1-\ell}\rho)^2) .
\end{align*}
The nondecreasing property of $t\mapsto t^{-\alpha}\eta(t)$ and the choice of $\sigma$ implies
\begin{align*}
 \sum_{\ell=0}^{k-1}\left(\frac{1}{2}\theta^{2\sigma}\right)^\ell\eta\left((\theta^{k-1-\ell}\rho)^2\right) 
 \leq &\theta^{-2\alpha}\eta\left((\theta^k\rho)^2\right)\sum_{\ell =0}^{k-1}
  \left(\frac{1}{2}\theta^{2\sigma-2\alpha}\right)^\ell \notag\\
 \leq &\frac{2\eta\left((\theta^k\rho)^2\right)}{2\theta^{2\alpha}-\theta^{2\sigma}}.
\end{align*}
Therefore we have
\begin{equation}
 \Phi(\theta^k\rho)\leq \left(\frac{1}{2}\theta^{2\sigma}\right)^k\Phi(\rho)
  +C_5\eta\left((\theta^k\rho)^2\right). \label{iteration}
\end{equation}
Keeping in mind of (b), (c) and the choice of $\rho$, we prove ($\I_k$). We next want to show ($\II_k$).
Using the fact that $\dashint_{B_\rho(x_0)}\nu(x-x_0)dx=0$ holds for all $\nu\in\mathrm{Hom}(\R^n,\R^N)$, 
H\"{o}lder's inequality and the Poincar\'{e} inequality, we obtain
\begin{align*}
 \lvert u_{x_0,\theta^k\rho}\rvert\leq &\lvert u_{x_0,\theta^{k-1}\rho}\rvert+
  \left\lvert\dashint_{B_{\theta^k\rho}(x_0)}(u-u_{x_0,\theta^{k-1}\rho}-(Du)_{x_0,\theta^{k-1}\rho}(x-x_0))dx\right\rvert \\
 \leq &\lvert u_{x_0,\theta^{k-1}\rho}\rvert+
  \theta^{-n/2}\sqrt{C_P}(1+\lvert(Du)_{x_0,\theta^{k-1}\rho}\rvert)\Phi^{1/2}(\theta^{k-1}\rho) \\
 \leq &\lvert u_{x_0,\rho}\rvert+
  \theta^{-n/2}\sqrt{C_P}\sum_{\ell =0}^{k-1}(1+\lvert(Du)_{x_0,\theta^\ell\rho}\rvert)\Phi^{1/2}(\theta^\ell\rho) .
\end{align*}
Similarly we see
\begin{align*}
 \lvert(Du)_{x_0,\theta^k\rho}\rvert \leq &\lvert(Du)_{x_0,\theta^{k-1}\rho}\rvert+
  \left\lvert\dashint_{B_{\theta^k\rho}(x_0)}(Du-(Du)_{x_0,\theta^{k-1}\rho})dx\right\rvert \\
 \leq &\lvert(Du)_{x_0,\rho}\rvert+
 \theta^{-n/2}\sum_{\ell =0}^{k-1}(1+\lvert(Du)_{x_0,\theta^\ell\rho}\rvert)\Phi^{1/2}(\theta^\ell\rho) .
\end{align*}
Combining two estimates and using (\ref{iteration}) and (\ref{sumeta}) we infer 
\begin{align*}
 &\lvert u_{x_0,\theta^k\rho}\rvert+\lvert(Du)_{x_0,\theta^k\rho}\rvert \\
 \leq &\lvert u_{x_0,\rho}\rvert+\lvert(Du)_{x_0,\rho}\rvert
  +\frac{(1+\sqrt{C_P})(1+2T_0)}{\theta^{n/2}}\sum_{\ell =0}^{k-1}\Phi^{1/2}(\theta^\ell\rho) \\
 \leq &\lvert u_{x_0,\rho}\rvert+\lvert(Du)_{x_0,\rho}\rvert 
  +\frac{(1+\sqrt{C_P})(1+2T_0)}{\theta^{n/2}}\sum_{\ell=0}^{k-1}\left\{\left(\frac{1}{\sqrt{2}}\theta^\sigma\right)^\ell\sqrt{\Phi(\rho)}
  +\sqrt{C_5\eta(\theta^{2\ell}\rho^2)}\right\}  \\
 \leq &\lvert u_{x_0,\rho}\rvert+\lvert(Du)_{x_0,\rho}\rvert 
  +\frac{(1+\sqrt{C_P})(1+2T_0)}{\theta^{n/2}}\left\{\frac{\sqrt{2\Phi(\rho)}}{\sqrt{2}-\theta^\sigma}
  +\sqrt{\frac{C_5\alpha^2\beta^2 \widetilde{F}(\rho^2)}{4(1-\theta^{\alpha\beta})^2}}\right\} \\
 \leq & T_0+\frac{(1+\sqrt{C_P})(1+2T_0)}{\theta^{n/2}}\frac{\sqrt{2\Phi_0}}{\sqrt{2}-\theta^\sigma}
  +\frac{(1+\sqrt{C_P})(1+2T_0)}{\theta^{n/2}}
  \sqrt{\frac{C_5\alpha^2\beta^2 \widetilde{F}(\rho^2)}{4(1-\theta^{\alpha\beta})^2}} \\
 \leq & T_0+ \frac{1}{\sqrt{2}-\theta^\sigma}\frac{\theta^{n/2}}{2}
  +\frac{1}{2}T_0 \\
 \leq & 2T_0.
\end{align*}
This proves ($\II_k$).
By hypothesis (c), ($\II_k$), ($\eta$4), the definition of $H$ and \eqref{sc3}, we easily derive 
\begin{align*}
 &(1+\lvert(Du)_{x_0,\theta^k\rho}\rvert)\kappa (\lvert u_{x_0,\theta^k\rho}\rvert+\lvert(Du)_{x_0,\theta^k\rho}\rvert)\theta^k\rho \\
 \leq &H(2T_0,2T_0)\sqrt{\eta(\rho_0)} \\
 \leq &1.
\end{align*}
Thus, we prove ($\III_k$). 

We next want to prove that $(Du)_{x_0,\theta^k\rho}$ converges to some limit $\Lambda_{x_0}$ in 
$\mathrm{Hom}(\R^n,\R^N)$. Arguing as in the proof of ($\II_k$) we deduce for $k>j$
\begin{align}
 \lvert(Du)_{x_0,\theta^k\rho}-(Du)_{x_0,\theta^j\rho}\rvert
 \leq &\sum_{\ell =j+1}^k\lvert(Du)_{x_0,\theta^\ell\rho}-(Du)_{x_0,\theta^{\ell -1}\rho}\rvert \notag\\
 \leq &\sum_{\ell =j+1}^k\theta^{-n/2}(1+\lvert(Du)_{x_0,\theta^{\ell -1}\rho}\rvert)\sqrt{\Phi(\theta^{\ell -1}\rho)} \notag\\
 \leq &\frac{(1+2T_0)\sqrt{\theta^{2\sigma j}\Phi(\rho)}}{\theta^{n/2}(\sqrt{2}-\theta^\sigma)}
  +\frac{(1+2T_0)}{\theta^{n/2}}
  \sqrt{\frac{C_5\alpha^2\beta^2 \widetilde{F}(\theta^{2j}\rho^2)}{4(1-\theta^{\alpha\beta})^2}}. \label{cauchy}
\end{align}
Taking into account our assumption ($\eta$3) we see that 
$\{ (Du)_{x_0,\theta^k\rho}\}_k$ is a Cauchy sequence in $\mathrm{Hom}(\R^n,\R^N)$. Therefore the limit
\[
 \Lambda_{x_0}:=\lim_{k\to\infty}(Du)_{x_0,\theta^k\rho}
\]
exists and from \eqref{cauchy} we infer for $j\in\N\cup\{ 0\}$
\begin{align*}
 \lvert(Du)_{x_0,\theta^j\rho}-\Lambda_{x_0}\rvert 
 \leq &\lvert(Du)_{x_0,\theta^k\rho}-(Du)_{x_0,\theta^j\rho}\rvert+\lvert(Du)_{x_0,\theta^k\rho}-\Lambda_{x_0}\rvert \\
 \to & C_7\sqrt{\theta^{2\sigma j}\Phi(\rho)+\widetilde{F}(\theta^{2j}\rho^2)} \quad (\text{as}\ k\to\infty)
\end{align*}
where
\[
 C_7:=\frac{\sqrt{2}(1+2T_0)}{\theta^{n/2}}\sqrt{\frac{1}{(\sqrt{2}-\theta^\sigma)^2}
 +\frac{C_5\alpha^2\beta^2}{4(1-\theta^{\alpha\beta})^2}} .
\]
Combining this with \eqref{iteration}, and recalling the estimate \eqref{etaF} we arrive at
\begin{align*}
 \dashint_{B_{\theta^j\rho}(x_0)}\lvert Du-\Lambda_{x_0}\rvert^2dx 
 \leq &2(1+2T_0)\Phi(\theta^j\rho)+2\lvert(Du)_{x_0,\theta^j\rho}-\Lambda_{x_0}\rvert^2 \\
 \leq & C_8 \left\{\theta^{2\sigma j}\Phi(\rho)+\widetilde{F}(\theta^{2j}\rho^2)\right\}
\end{align*}
with
\[
 C_8:=2\left\{1+2T_0+{C_7}^2
 +\frac{C_5\alpha^2\beta^2 (1+2T_0)}{4(1-\theta^{\alpha\beta})^2}\right\} .
\]
For $0<r\leq \rho$ we find $j\in\N\cup\{ 0\}$ such that $\theta^{j+1}\rho\leq r \leq \theta^j\rho$. Then using the
above estimate with \eqref{Fnoninc} implies
\begin{align*}
 \dashint_{B_r(x_0)}\lvert Du-\Lambda_{x_0}\rvert^2dx 
 &\leq \theta^{-n}\dashint_{B_{\theta^j\rho}(x_0)}\lvert Du-\Lambda_{x_0}\rvert^2dx \\
 &\leq C_8\theta^{-n}\{\theta^{2\sigma j}\Phi(\rho)+\widetilde{F}(\theta^{2j}\rho^2)\} \\
 &\leq 4C_8\theta^{-n-2\sigma}\left\{\left(\frac{r}{\rho}\right)^{2\sigma}\Phi(\rho)+\widetilde{F}(r^2)\right\} .
\end{align*}
This proves \eqref{campanatosemi} with $C_6:=4C_8\theta^{-n-2\sigma}$. 
\end{proof}
The main theorem (Theorem \ref{pr}) is obtained from Lemma \ref{Iteration} by using standard arguments. \\
\mbox{}\\
{\bf Acknowledgments}\\
The author thanks Professor Hisashi Naito for helpful discussions.

\def\cprime{$'$}
\providecommand{\bysame}{\leavevmode\hbox to3em{\hrulefill}\thinspace}
\providecommand{\MR}{\relax\ifhmode\unskip\space\fi MR }
\providecommand{\MRhref}[2]{%
  \href{http://www.ams.org/mathscinet-getitem?mr=#1}{#2}
}
\providecommand{\href}[2]{#2}

\mbox{}\\
Taku Kanazawa \\
Graduate School of Mathematics \\
Nagoya University \\
Chikusa-ku, Nagoya, 464-8602, JAPAN \\
E-mail:taku.kanazawa@math.nagoya-u.ac.jp
\end{document}